\documentclass[a4paper,10pt,oneside]{amsart}
\usepackage{latexsym}
\usepackage{amsmath}
\usepackage{amsfonts}
\usepackage{amsthm}
\usepackage{amssymb}

\usepackage[utf8]{inputenc}

\def\cl{\text{cl}}

\def\cont{{2^{\aleph_0}}}

\def\conc{{^\smallfrown}}

\def\intr{\mbox{int}\ }
\def\cl#1{%
\overline{#1}
}

\swapnumbers

\newtheoremstyle{theorem}
  {}
  {}
  {\itshape}
  {}
  {\bfseries\scshape}
  {.}
  {.5em}
  {}
\theoremstyle{theorem}
\newtheorem*{theorem*}{Theorem}
\newtheorem{theorem}[subsection]{Theorem}

\newtheorem{proposition}[subsection]{Proposition}
\newtheorem{observation}[subsection]{Observation}

\newtheorem{definition}[subsection]{Definition}

\newtheorem*{ack}{Acknowledgements}
\newtheoremstyle{example}
  {}
  {}
  {}
  {}
  {\bfseries\scshape}
  {.}
  {.5em}
  {}
\theoremstyle{example}

\begin{document}
\thispagestyle{empty}
\title{Lonely points revisited}
\author{Jonathan L. Verner}
\address{
Department of Logic, Faculty of Arts, Charles University \\ Palachovo nám. 2, 116 38 Praha 1 \\ Czech Republic}
\email{jonathan.verner{@}matfyz.cz}
\thanks{The author would like to acknowledge the support of GAČR 401/09/H007 Logické základy sémantiky}

\begin{abstract}
In our previous paper, Lonely points, we introduced the notion of a lonely point, due to P. Simon. A point $p\in X$ is lonely if it is a limit point
of a countable dense-in-itself set, not a limit point a countable discrete set and all countable sets whose limit point it is, form a filter. 
We use the space ${\mathcal G}_\omega$ from a paper of A. Dow, A.V. Gubbi and A. Szymański (\cite{DGS88}) to construct lonely points in $\omega^*$. 
This answers the question of P. Simon posed in our paper Lonely points (\cite{Ver08}).
\end{abstract}
\maketitle

\section{Introduction}

\begin{definition} A \emph{topological type} in a space $X$ is a subset $T\subseteq X$ which is invariant under homeomorphisms.
\end{definition}

An example of a topological type are the discrete points in a space $X$. Another more interesting type is given in the following
definition. The first part is due to W. Rudin (\cite{Rudin56}) the second to K. Kunen (\cite{Kunen78}).

\begin{definition}[Rudin, Kunen] A point $x\in X$ is a \emph{P-point} if the countable intersection of neighbourhoods of $x$ is again a neighbourhood of $x$.
It is a \emph{weak P-point} if it is not a limit point of a countable subset of $X$.
\end{definition}

Clearly any isolated point is a P-point, and a P-point is a weak P-point. However none of the implications can be reversed.

If a space contains two distinct topological types, then it is not homogeneous. The motivation for finding topological 
types in $\omega^*$ was given by the following surprising result of Z. Frolík (\cite{Frolik67a},\cite{Frolik67b}):

\begin{theorem}[Frolík] $\omega^*$ is not homogeneous.
\end{theorem}

His proof used a clever combinatorial argument but it gave no intrinsically topological reason for the non-homogeneity of $\omega^*$. This
motivated the question whether one can find a ``topologically defined'' topological type --- an ``honest'' proof of nonhomogeneity. 
Under CH, this was answered already by W.~Rudin in \cite{Rudin56} where he proved that P-points exist in $\omega^*$. 
However in ZFC the question remained open for some twenty years.

In his seminal paper \cite{Kunen78}, K. Kunen proved in ZFC that $\omega^*$ contains a weak P-point:

\begin{theorem}[Kunen] $\omega^*$ contains a weak P-point.
\end{theorem}

Since it obviously contains non weak P-points, this is an ``honest'' proof of nonhomogeneity.
In \cite{Mill82}, J. van Mill had exploited the techniques of K. Kunen to prove, in ZFC, the existence of sixteen distinct topological types in $\omega^*$!
One of the types he introduced is given in the following theorem:

\begin{theorem}[van Mill] There is a point $p\in\omega^*$ which is a limit point of a countable discrete set and the countable sets whose limit point it is
form a filter.
\end{theorem}
\begin{proof}(Idea) Use Kunen's result to construct a weak P-point $p\in\omega^*\subseteq\beta\omega$. Now use a theorem of P. Simon (see theorem \ref{embed}) 
to embed $\beta\omega$ into $\omega^*$ as a weak P-set. Then the image of $p$ via the embedding will be as required, since $p$ clearly has the property in 
$\beta\omega$ and the embedding does not destroy it since the image of $\beta\omega$ is a weak P-set.
\end{proof}

This motivated P. Simon to define the following notion, which we have called a lonely point in \cite{Ver08}. We want essentially the same type of point as in
the above theorem only replacing the countable discrete set whose limit point it is by a crowded set:

\begin{definition} A point $p\in X$ is a \emph{lonely point} provided:
 \begin{itemize}
  \item[(i)] $p$ is $\omega$-discretely untouchable, i.e. not a limit point of a countable discrete set,
  \item[(ii)] $p$ is a limit point of a countable crowded (i.e. without isolated points) set and
  \item[(iii)] the countable sets whose limit point $p$ is form a filter.
 \end{itemize}
\end{definition}

In the paper we were able to show that lonely points exist in some open dense subspace of $\omega^*$.
Here we prove that they actually exist in $\omega^*$:

\begin{theorem} $\omega^*$ contains a lonely point.
\end{theorem}

The idea is to construct a countable, perfectly disconnected space $X$ with an $\aleph_0$-bounded remainder and then embed it as a weak P-set into
$\omega^*$. Any point of $X$ will then be a lonely point of $\beta X$ and, since $\beta X$ will be a weak P-set in $\omega^*$, also a lonely point
of $\omega^*$.

\section{Basic definitions and theorems}

\begin{definition}[Kunen] $F\subseteq X$ is a weak P-set of $X$ if any countable $D\subseteq X$ disjoint from $F$ has closure disjoint from $F$.
\end{definition}

\begin{observation}\label{hereditarylonely} If $F\subseteq X$ is a weak P-set of $X$ and $x\in F$ is a lonely point of $F$ then it is also a lonely point of $X$.
\end{observation}

\begin{definition} A space $X$ is \emph{extremally disconnected} (or ED for short) if the closure of any open set is open.
\end{definition}

The following is standard, see e.g. \cite{Engelking}:

\begin{theorem} If $X$ is ED then so is $\beta X$
\end{theorem}

We shall also need the following theorem of P. Simon (see \cite{Sim85}):

\begin{theorem}[Simon]\label{embed} The \v{C}ech-Stone compactification of any $T_3$ ED space of weight $\leq\cont$ can be embedded into $\omega^*$ as a closed weak $P$-set.
\end{theorem}

\section{Irresolvable spaces}

In this section, unless otherwise stated, we assume all spaces to be crowded (i.e. without isolated points). The following definitions
were introduced in \cite{vD93}:

\begin{definition}[van Douwen] A crowded space $X$ is \emph{perfectly disconnected} if no point of $X$ is a limit point of two disjoint subsets of $X$. It is irresolvable,
if it contains no disjoint dense sets. It is open-hereditarily-irresolvable (OHI for short), provided each open subspace is irresolvable.
A crowded space is \emph{maximal regular} if each finer topology either contains an isolated point or is not regular.
\end{definition}

Irresolvable spaces were constructed by E. Hewitt (\cite{Hew43}) and independently by M. Katětov (\cite{Katetov47}). They were extensively studied in 
\cite{vD93} where the following theorems may be found:


\begin{theorem}[\cite{vD93},1.7,1.11]\label{maxreg} Maximal regular spaces are zerodimensional, ED and OHI.
\end{theorem}
\begin{theorem}[\cite{vD93},1.4,1.6]\label{ultradisc} If $A,B$ are disjoint crowded subspaces of a maximal regular space, then $\cl{A}$ and $\cl{B}$ are disjoint.
\end{theorem}
\begin{theorem}[\cite{vD93},2.2]\label{perfdisc_char} If $X$ is ED and OHI and each nowhere dense subset of $X$ is closed then $X$ is perfectly disconnected.
\end{theorem}

The following theorem is not explicitly stated in van Douwen's paper, but its proof is essentially given in his Lemma 3.2 and Example 3.3.

\begin{theorem}[van Douwen]\label{perfdisc} Any countable maximal regular space $X$ contains an open perfectly disconnected subspace.
\end{theorem}
\begin{proof} For each $Z\subseteq X$ let 
\begin{displaymath}
 A_Z=\{x\in Z:x\ \mbox{is a limit point of a relatively discrete subset of Z}\}
\end{displaymath}

{\bf Claim} $A_Z\neq Z$ for each open subset $Z$ of $X$.

Assume otherwise. Enumerate $Z$ as $\langle x_n:n<\omega\rangle$. By induction construct pairwise disjoint, relatively discrete sets $\langle D_n:n<\omega\rangle$
such that:
\begin{itemize}
 \item[(i)] $\bigcup_{i< n} D_i\subseteq \cl{D_n}$ for all $n<\omega$ and
 \item[(ii)] $x_n\in\cl{D_n}$ for $n<\omega$.
\end{itemize}
This will lead to a contradiction with the irresolvability of $Z$ (by theorem \ref{maxreg}, $X$ is OHI, so $Z$ is irresolvable) since $\bigcup_{n<\omega}D_{2n}$ and $\bigcup_{n<\omega}D_{2n+1}$ would then be disjoint 
dense subsets of $Z$. To see that the construction can be carried out let $D_0=\{x_0\}$ and assume we have constructed $D_i$ for $i\leq n$. 
Let $Y=D_n\cup Z\setminus\cl{D_n}$. Since $D_n$ is relatively discrete, $Y$ is open. Since $Z$ is regular and $D_n$ is countable and relatively discrete, 
there is a pairwise disjoint collection of open sets $\{U_x:x\in D_n\}$ such that $x\in U_x\subseteq Y$. Since we assumed $A_Z=Z$ we can choose for each $x\in D_n$
a relatively discrete set $D_x$ such that $D_x\subseteq U_x$ and $x\in\cl{D_x}\setminus D_x$. Let $D^\prime_{n+1}=\bigcup_{x\in D_n}D_x$. If $x_{n+1}$ is a limit
point of $D^\prime_{n+1}$ let $D_{n+1}=D^\prime_{n+1}$ otherwise let $D_{n+1}=D^\prime_{n+1}\cup\{x_{n+1}\}$. Then $D_{n+1}$ is as required.

{\bf Claim} $\intr A_X=\emptyset$.

For any clopen $U$, $A_X\cap U  = A_U$. Since $X$ is regular and countable, it is zerodimensional. Suppose $U$ is clopen and $U\subseteq A_X$. 
By the previous claim $U\setminus A_U\neq\emptyset$ but then $U\setminus A_X\neq\emptyset$ a contradiction.

{\bf Claim} $A_X$ is nowhere dense.

Take any open $U\subseteq X$. Then $U\setminus A_X$ is dense in $U$, since $\intr A_X=\emptyset$. Since
$X$ is OHI (by theorem \ref{maxreg}), $U$ is irresolvable so $A_X$ cannot be dense in $U$ so $U\not\subseteq\cl{A_X}$. Thus $\intr\cl{A_X}=\emptyset$.

{\bf Claim} If $A\subseteq X$ is nowhere dense then there is a discrete $D\subseteq A$ dense in $A$.

Let $D=\{x\in A:x\ \mbox{is isolated in}\ A\}$. Since $X$ is regular and countable $D$ is relatively discrete. Since $A$ is nowhere dense, 
$D$ is discrete. Let $E=A\setminus\cl{D}$. Then $E$ has no isolated points. Also $X\setminus E$ has no isolated points. By theorem \ref{ultradisc}
$E$ must be open which contradicts that $A$ is nowhere dense.

Let
\begin{displaymath}
 \vartheta=\{x\in X:x\ \mbox{is not a limit point of a nowhere dense subset of}\ X\}
\end{displaymath}

By the previous claim (and by the fact that each discrete subset of $X$ is nowhere dense)

\begin{displaymath}
\vartheta = \{x\in X:x\ \mbox{is not a limit point of a discrete set}\}
\end{displaymath}

Then $X\setminus\vartheta\subseteq A_X$ so $X\setminus\vartheta$ is nowhere dense, so $\intr{\vartheta}$ is nonempty. 
We finally show that $\intr{\vartheta}$ is perfectly disconnected. By the definition of $\vartheta$ any 
nowhere dense subset of $\intr{\vartheta}$ is closed. Now it remains to apply theorem \ref{perfdisc_char} remembering that by theorem \ref{maxreg}
$\intr{\vartheta}$ is ED and OHI (any open subspace of a maximal regular space is maximal regular).
\end{proof}

\section{Proof of the main theorem}

The following definition and theorem is taken from \cite{DGS88}:

\begin{definition} Let $p\in\omega^*$ be a weak P-point. The space ${\mathcal G}_\omega$ is the space $\omega^{<\omega}$ of all finite
sequences of natural numbers with $G\subseteq\omega^{<\omega}$ being open precisely when for each $\sigma\in G$ the set $\{n:\sigma\conc n\in G\}$ is
in $p$.
\end{definition}

\begin{theorem}[Dow, Gubbi, Szymanski]\label{bounded} The remainder of ${\mathcal G}_\omega$ is $\aleph_0$-bounded. Moreover ${\mathcal G}_\omega$ is a $T_2$, zerodimensional, ED space.
\end{theorem}

Notice that if a space $X$ has an $\aleph_0$-bounded remainder, any finer topology also has an $\aleph_0$-bounded remainder:

\begin{proposition} If $(X,\tau)^*$ is a zerodimensional $\aleph_0$-bounded space and $\sigma\supseteq\tau$ is also zerodimensional, then $(X,\sigma)^*$ is $\aleph_0$-bounded.
\end{proposition}
\begin{proof} Note that any $p\in (X,\tau)^*$ corresponds to a closed subset of $(X,\sigma)^*$ (denote it $[p]$). Now given $\{q_n:n<\omega\}\subseteq(X,\sigma)^*$ we
can find $\{p_n:n<\omega\}\subseteq (X,\tau)^*$ such that $\{q_n:n<\omega\}\subseteq\bigcup\{ [p_n]:n<\omega\}$. Since $(X,\tau)^*$ is $\aleph_0$-bounded,
$\cl{\{p_n:n<\omega\}}^{\beta (X,\tau)}\cap X=\emptyset$ so also $\cl{\{q_n:n<\omega\}}^{\beta(X,\sigma)}\cap X=\emptyset$ which implies that $(X,\sigma)^*$ is
$\aleph_0$-bounded.
\end{proof}

\begin{theorem}\label{compactlonely} There is a countable, ED, perfectly disconnected space $X$ with an $\aleph_0$-bounded remainder.
\end{theorem}
\begin{proof} Take the space ${\mathcal G}_\omega$ from theorem \ref{bounded}, and refine the topology to a maximal regular topology. Then, by the previous
 proposition, this space still has an $\aleph_0$-bounded remainder and so does its open perfectly disconnected subspace given by theorem \ref{perfdisc}. Let
$X$ be this subspace.
\end{proof}

\begin{theorem} $\omega^*$ contains a lonely point.
\end{theorem}
\begin{proof}
 Let $X$ be the space from the previous theorem. Since it is crowded perfectly disconnected, each of its points is a lonely point of $X$. Since its remainder
is $\aleph_0$-bounded, each of its points is also a lonely point of $\beta X$. Since it is ED, $\beta X$ is also ED and since it is countable, $\beta X$ has
weight at most $\cont$. Hence, by theorem \ref{embed}, $\beta X$ can be embedded as a weak P-set into $\omega^*$ and each point of $X$ will be a lonely point of $\omega^*$
(by observation \ref{hereditarylonely}).
\end{proof}

\begin{ack} The author would like to thank A. Dow for a stimulating discussion about the topic as well as members of the Prague Set Theory seminar, who gave
numerous helpful comments and encouraged him to work on the topic. The final version of the paper was written while the author was visiting the Kurt G\"odel Research Center 
for Mathematical Logic and he would like to thank its members for their hospitality.
\end{ack}

\bibliographystyle{plain}
\bibliography{lonely-revisited}
\medskip

\end{document}